\documentclass[12pt,draft]{article}

\usepackage{amsmath,amsfonts,amsthm}
\usepackage{amssymb}
\usepackage{euscript}
\usepackage{color}

\newtheorem{proposition}{Proposition}
\newtheorem{theorem}{Theorem}

\theoremstyle{definition}

\theoremstyle{definition}

\title{Orthogonally Additive Sums of Powers\\ of Linear Functionals}

\author{Christopher Boyd, Raymond Ryan \& Nina Snigireva}

\date{}

\makeatletter
        \def\@thefnmark{\null}
        \def\footnotetexta{\@footnotetext}

\begin{document}
\baselineskip=.65cm 
\maketitle

\begin{abstract}\footnotetexta{{\bf Keywords:} Banach lattice;  Orthogonally
    additive polynomial.}
\footnotetexta{{\bf MSC(2020):} 46G25; 46B42; 46E05.}
  Let $E$ be a Banach lattice, $\lambda_1,\lambda_2,\ldots,\lambda_k$ 
  non-zero scalars and $\varphi_1,\varphi_2,\ldots,\varphi_k$ pairwise
  independent linear functionals on $E$. We
  show that if $k<m$ then $\sum_{j=1}^k\lambda_j\varphi_j^m$ is orthogonally
  additive if and only if $\varphi_j$ or $-\varphi_j$ is a lattice homomorphism
  for each $j$, $1\le j\le k$. Moreover, for each $m\ge 2$, we provide an
  example to
  show that this result does not extend to the case where $k=m$.
\end{abstract}

\section{Introduction}

Let $E$ be a Banach space. An $m$-homogeneous polynomial $P\colon E\to
\mathbb{R}$ is defined by $P(x)=A(x,\dots,x)$, where $A$ is a bounded symmetric
$m$-linear mapping from $ E^m$ to $\mathbb{R}$. We write $P=\hat A$.
The supremum norm on the space of $m$-homogeneous polynomials is given by
$\|P\|=\sup_{\|x\|\leq 1}|P(x)|$. For further information on polynomials we
refer the reader to \cite{Dinb}.

When $E$ is a Banach lattice we say that an $m$-homogeneous polynomial
$P\colon E\to\mathbb{R}$ with associated symmetric $m$-linear mapping $A$ is
positive if ${A}(x_1,\ldots,x_n)\ge 0$ for all $x_1,\ldots,x_n\ge 0$ in $E$. An
$m$-homogeneous polynomial $P\colon E\to
\mathbb{R}$ is said to be regular if it can be written as the difference of two
positive polynomials.

A $m$-homogeneous polynomial $P\colon E\to\mathbb{R}$ is said to be
orthogonally additive if $P(x+y)=P(x)+P(y)$ whenever $|x|\wedge |y|=0$.
Orthogonally additive polynomials were introduced by Sunderesan \cite{Sun}
and have been an area of active research in recent years. See \cite{ABSV,BLL,
  CLZ,K,PPV}. If $A$
is the symmetric $m$-linear mapping associated with $P$ then it is shown in
\cite[Lemma~4.1]{BB} that $P$ is
orthogonally additive if and only if $A(x_1,\dots,x_m)=0$ whenever some pair,
$x_i$,
$x_j$, of the arguments are disjoint ($|x_i|\wedge |x_j| = 0$ for $i\not=j$).
An $m$-linear mapping with this property is said to be  orthosymmetric.

Let $E$ and $F$ be Banach lattices. A bounded linear operator $T\colon E\to F$
is said to be a lattice homomorphism if it preserves the lattice operations,
namely, $T(x\vee y)=(Tx)\vee (Ty)$ and $T(x\wedge y)=(Tx)\wedge(Ty)$ for all
$x$ and $y$ in $E$.

We have the following characterisation of lattice homomorphisms which we will
use in this paper.

\begin{proposition}\cite[Proposition~1.3.11]{MN}\label{Charhomo}
  Let $E$ and $F$ be Banach lattices and $T\colon E\to F$ be a bounded linear
  operator. The following are equivalent:
  \begin{enumerate}
  \item[(a)] $T$ is a lattice homomorphism.
  \item[(b)] $|Tx|=T|x|$ for all $x$ in $E$.
  \item[(c)] $T^+x\wedge T^-x=0$ for all $x$ in $E$.
  \end{enumerate}
\end{proposition}

For more information on Banach lattices and operators on Banach lattices we
refer the reader to \cite{AB,MN}.

\section{Sums of Powers of Linear Functionals}
It is shown in \cite[Proposition~2]{BRS} that if $E$ is a Banach lattice,
$m>1$ a
positive integer and $\varphi$  a (continuous) linear functional then
$\varphi^
m$ is orthogonally additive if and only if $\varphi$ or $-\varphi$ is a 
homomorphism. Let us
now generalise this result to sums of of powers of linear functionals. We begin
with the following proposition.

\begin{proposition}\label{deriv}
  Let $E$ be a Banach lattice, $k$, $m$ positive integers with $k<m$ and
  $P$ an orthogonally additive $m$-homogeneous polynomial on $E$. Then for
  every $x$ in $E$,
  ${\hat d}^{k}P(x)$ is an orthogonally additive $k$-homogeneous polynomial on
  $E$.
\end{proposition}

\begin{proof} As $P$ is orthogonally additive the symmetric $m$-linear mapping
  $A$ associated with $P$, is orthosymmetric. Since the symmetric
  $k$-linear mapping associated with ${\hat d}^{k}P(x)$ is $\frac{m!}{(m-k)
    !}A(\underbrace{x,\ldots,x}_{m-k},\dots)$  we see
  that ${\hat d}^{k}P(x)$ is an orthogonally additive $k$-homogeneous
  polynomial.
\end{proof}  

In the statement of the following theorem we require a $m$-homogeneous
polynomial written in the form $\sum_{j=1}^k\lambda_j\varphi_j^m$  with
$\varphi_i
\not=\lambda\varphi_j$. We note that given any $m$ homogeneous polynomial of 
the form $\sum_{j=1}^k\lambda_j \varphi_j^m$ we can amalgamate terms and rewrite
it as $\sum_{i=1}^l\mu_i\psi_i^m$ with $l\le k$ where $\psi_1,\ldots,\psi_l$ are
pairwise independent.

\begin{theorem}\label{hom}
  Let $E$ be a Banach lattice and let $k$, $m$ be positive integers with $k<m$.
  Let $\lambda_1,\lambda_2,\ldots,\lambda_k$ be non-zero scalars and let
  $\varphi_1,
  \varphi_2,\ldots,\varphi_k$ be pairwise independent linear functionals on
  $E$ ($\varphi_i\not=\lambda\varphi_j$ for $i\not= j$). Then $P=\sum_{j=1}^k
\lambda_j
  \varphi_j^m$ is
  orthogonally additive if and only if $\varphi_j$ or $-\varphi_j$ is a lattice
  homomorphism for each $j$, $1\le j\le k$.
\end{theorem}

\begin{proof} If for each $j$, $\varphi_j$ or $-\varphi_j$ is a lattice
  homomorphism then $P$ is orthogonally additive. 

  To prove the converse we use induction on $m$. By
  \cite[Proposition~2]{BRS} the result is true when $k=1$ and $m=2$.

  Let us now suppose that the result is true for $m$ and every $k$ with $k<m$.
  We will now establish the result for every $l<m+1$.

  Suppose that $P=\sum_{j=1}^{l}\lambda_j\varphi_j^{m+1}$ is an orthogonally
  additive polynomial.
  Given $1\le r\le l$ choose $s$, $1\le s\le l$, so that $s\not= r$. Since
  $\varphi_r$ and
  $\varphi_s$ are linearly independent, we can choose $x$ in $E$ so that
  $\varphi_r(x)\not=0$ and $\varphi_s(x)=0$. Using Proposition~\ref{deriv}, we
  have that ${\hat d}^mP(x)$ is orthogonally additive. 
  
  Then
  \begin{align*}
    {\hat d}^mP(x)&=(m+1)!\sum_{j=1}^{l}\lambda_j\varphi_j(x)\varphi_j^m\\
                 &=(m+1)!\sum_{j=1,j\not=s}^{l}\lambda_j\varphi_j(x)\varphi_j^m
  \end{align*}
  is a sum of at most $l-1$ linear functionals each raised to the power of $m$.
  Hence, by our induction hypothesis, we have that for each $j$ with
  $\varphi_j(x)\not=0$ either  $\varphi_j$ or $-\varphi_j$ is a homomorphism.
  In particular, we get that either $\varphi_r$ or $-\varphi_r$
  is a homomorphism. As $r$ was arbitrary we conclude that $\varphi_j$ or
  $-\varphi_j$ is a homomorphism for $1\le j\le l$.
\end{proof}

We claim that the above theorem is sharp. This is particularly easy to see
when $k=m=2$. Let us take $E$ to be $\mathbb{R}^2$ with the standard order.
 Let $\varphi_1(x)=x_1+x_2$ and $\varphi_2
 (x)=x_1-x_2$. It follows from Proposition~\ref{Charhomo} that neither
 $\varphi_1$ or $\varphi_2$ are homomorphisms.
However,
$$
\varphi_1^2(x)+\varphi_2^2(x)=2(x_1^2+x_2^2)
$$
is orthogonally additive.

 More generally, let us show that for
each integer $m$ we can find linear functionals $\varphi_1,\ldots,\varphi_m$ on
$\mathbb{R}^2$, none of which or their negatives are homomorphisms, such
that $\sum_{j=1}^m
\lambda_j\varphi_j^m$ is orthogonally additive.

We will see that it is possible to choose $A_1,\ldots, A_n$ so that
$$
\sum_{r=1}^nA_r\left((rx_1+x_2)^{2n}+(rx_1-x_2)^{2n}\right)=B_1x_1^{2n}+B_2x_2
^{2n}.
$$

Let us first suppose that $m=2n$ is an even integer. Then we have
\begin{align*}
  (x_1+x_2)^{2n}+(x_1-x_2)^{2n}&= 2\sum_{j=0}^n{2n\choose 2(n-j)}x_1^{2(n-j)}
                                 x_2^{2j}\\
  (2x_1+x_2)^{2n}+(2x_1-x_2)^{2n}&= 2\sum_{j=0}^n{2n\choose 2(n-j)}2^{2(n-j)}
                                   x_1^{2(n-j)}x_2^{2j}\\
                               &\vdots\\
  (nx_1+x_2)^{2n}+(nx_1-x_2)^{2n}&=2\sum_{j=0}^n{2n\choose 2(n-j)}n^{2(n-j)}
                                   x_1^{2(n-j)}x_2^{2j}.
\end{align*}

Adding we get that
\begin{align*}
  \sum_{r=1}^nA_r&\left((rx_1+x_2)^{2n}+(rx_1-x_2)^{2n}\right)\\
 &=\sum_{r=1}^n
  A_r\left(2\sum_{j=0}^n{2n\choose 2(n-j)}r^{2(n-j)}x_1^{2(n-j)}x_2^{2j}\right)
  \\
 &=2\sum_{j=0}^n{2n\choose 2(n-j)}\left(\sum_{r=1}^n A_r r^{2(n-j)}\right)
  x_1^{2(n-j)}x_2^{2j}.
\end{align*}

We now show that it is possible to choose $A_1,\ldots, A_n$
so that
$$
\sum_{r=1}^n A_r r^{2n}\not=0,
$$
and
$$
\sum_{r=1}^n A_r r^{2(n-j)}=0
$$
for $j=1,\ldots,n-1$.

Without loss of generality let us suppose we wish to have $\sum_{r=1}^nA_r 
r^{2n}=1$.
We rewrite these equations in matrix form as
$$
\begin{pmatrix}
  1 & 2^2 & 3^2 &\ldots & n^2\\
  1 & 2^4 & 3^4 &\ldots & n^4\\
  \vdots &\vdots &\vdots & \ddots & \vdots\\
    1 & 2^{2(n-1)} & 3^{2(n-1)} &\ldots & n^{2(n-1)}\\
  1 & 2^{2n} & 3^{2n} &\ldots & n^{2n}
\end{pmatrix}\begin{pmatrix}A_1\\ A_2\\ \vdots\\ A_{n-1}\\ A_n\end{pmatrix}
  =\begin{pmatrix}0\\ 0\\ \vdots\\ 0\\ 1\end{pmatrix}.
$$  
The transpose of the matrix
$$
\begin{pmatrix}
  1 & 2^2 & 3^2 &\ldots & n^2\\
  1 & 2^4 & 3^4 &\ldots & n^4\\
  \vdots &\vdots &\vdots & \ddots & \vdots\\
    1 & 2^{2(n-1)} & 3^{2(n-1)} &\ldots & n^{2(n-1)}\\
  1 & 2^{2n} & 3^{2n} &\ldots & n^{2n}
\end{pmatrix}
$$
is
$$
\begin{pmatrix} 1 & 1 &\ldots & 1\\
   2^{2} & 2^{4} &\ldots & 2^{2n}\\
   3^{2} & 3^{4} &\ldots & 3^{2n}\\
  \vdots &\vdots & \ddots & \vdots\\
  (n-1)^2 & (n-1)^4 &\ldots & (n-1)^{2n}\\
  n^2 & n^4  &\ldots & n^{2n}
\end{pmatrix}
$$
As this matrix has the same determinant as the Vandermonde matrix
$$
\begin{pmatrix} 1 & 0 & 0 &\ldots &0\\
 1 & 1 & 1 &\ldots & 1\\
 1 & 2^{2} & 2^{4} &\ldots & 2^{2n}\\
 1 &  3^{2} & 3^{4} &\ldots & 3^{2n}\\
 \vdots &  \vdots &\vdots & \ddots & \vdots\\
 1 & (n-1)^2 & (n-1)^4 &\ldots & (n-1)^{2n}\\
 1 & n^2 & n^4  &\ldots & n^{2n}
\end{pmatrix},
$$
it is invertible. Therefore we can find $A_1,\ldots, A_n$ so that
\begin{equation}\label{long}
\sum_{r=1}^nA_r\left((rx_1+x_2)^{2n}+(rx_1-x_2)^{2n}\right)=x_1^{2n}+B_2x_2
^{2n}
\end{equation}
and hence $\sum_{r=1}^nA_r\left((rx_1+x_2)^{2n}+
  (rx_1-x_2)^{2n}\right)$ is orthogonally additive.

Taking the $(2n-1)^{st}$ derivative of both sides of (\ref{long}) at the
point $(1,1)$  we get that
$$
  \sum_{r=1}^nA_r\left((r+1)(rx_1+x_2)^{2n-1}+(r-1)(rx_1-x_2)^{2n-1}\right)= 
x_1^{2n-1}+B_2x_2^{2n-1},
$$
or equivalently
\begin{align*}
2A_1(x_1+x_2)^{2n-1}&+\sum_{r=2}^nA_r\left((r+1)(rx_1+x_2)^{2n-1}
  +(r-1)(rx_1-x_2)^{2n-1}\right)\\
                                                  &= x_1^{2n-1}+B_2x_2^{2n-1},
\end{align*}
giving us an orthogonally additive sum of $2n-1$ linear functionals, each raised
to the power of $2n-1$, none of which, nor their negatives, by
Proposition~\ref{Charhomo}, is a homomorphism.

\noindent Christopher Boyd, School of Mathematics \& Statistics, University
College Dublin, \hfil\break
Belfield, Dublin 4, Ireland.\\
e-mail: christopher.boyd@ucd.ie

\medskip

\noindent Raymond A. Ryan, School of Mathematics, Statistics and Applied 
Mathematics, National University of Ireland Galway, Ireland.\\
e-mail: ray.ryan@nuigalway.ie

\medskip

\noindent Nina Snigireva, School of Mathematics \& Statistics, University
College Dublin, \hfil\break
Belfield, Dublin 4, Ireland.\\
e-mail: nina@maths.ucd.ie
\end{document}